\documentclass[11pt]{article}
\usepackage{amsmath, amsthm, amsfonts}
\usepackage{latexsym}
\usepackage{graphicx, psfrag}
\makeatletter
\newfont{\footsc}{cmcsc10 at 8truept}
\newfont{\footbf}{cmbx10 at 8truept}
\newfont{\footrm}{cmr10 at 10truept}
\makeatother
\newtheorem{theorem}{\bf Theorem}
\newtheorem{proposition}{\bf Proposition}
\newtheorem{lemma}{\bf Lemma}

\newtheorem{corollary}{\bf Corollary}

\begin{document}
\title{On Stochastic Comparisons of Order Statistics from Heterogeneous Exponential Samples}

\author{Yaming Yu\\
\small Department of Statistics\\[-0.8ex]
\small University of California\\[-0.8ex] 
\small Irvine, CA 92697, USA\\[-0.8ex]
\small \texttt{yamingy@uci.edu}}

\date{}
\maketitle

\begin{abstract}
We show that the $k$th order statistic from a heterogeneous sample of $n\geq k$ exponential random variables is larger than that from a homogeneous exponential sample in the sense of star ordering, as conjectured by Xu and Balakrishnan (2012).  As a consequence, we establish hazard rate ordering for order statistics between heterogeneous and homogeneous exponential samples, resolving an open problem of P\v{a}lt\v{a}nea (2008).  Extensions to general spacings are also presented.

\end{abstract}

\section{Introduction and main results}
There exists a large literature on stochastic comparisons between order statistics arising from possibly heterogeneous populations, see Balakrishnan and Zhao (2013) for a review.  In reliability theory, order statistics play a prominent role as the lifetime of a $k$-out-of-$n$ system can be represented by the $n-k+1$th order statistic of the $n$ component lifetimes.  Because of the complexity of the distribution of order statistics arising from heterogeneous populations, stochastic comparisons with those from a homogeneous population are helpful, and can provide bounds on tail probabilities and hazard rates.  Of particular interest is the following result of Bon and P\v{a}lt\v{a}nea (2006).  Let $X_1,\ldots, X_n$ be independent exponential random variables with rates $\lambda_1,\ldots, \lambda_n$ respectively.  Let $Y_1,\ldots, Y_n$ be another set of independent exponential random variables with a common rate $\gamma$.  Then the $k$th order statistic of the heterogeneous sample is larger than that of the homogeneous sample in the usual stochastic order, that is, $Y_{k:n}\leq_{\rm st} X_{k:n}$, if and only if 
\begin{equation}
\label{paltanea}
\gamma\geq \left({n \choose k}^{-1} s_k(\lambda_1,\ldots, \lambda_n)\right)^{1/k},
\end{equation}
where $s_k$ denotes the $k$th elementary symmetric function.  For the sample maximum, Khaledi and Kochar (2000) showed that (\ref{paltanea}) with $k=n$ implies that $X_{n:n}$ is larger than $Y_{n:n}$ according to both the hazard rate order and the dispersive order.  In the context of failsafe systems, P\v{a}lt\v{a}nea (2008) showed that (\ref{paltanea}) with $k=2$ is equivalent to $Y_{2:n}\leq_{\rm hr} X_{2:n}$.  In the $k=2$ case, Zhao et al. (2009) obtained related results for the likelihood ratio ordering. 

In this paper we aim to extend such results to general $k$.  Specifically, we have 

\begin{theorem}
\label{thm1}
For each $1\leq k\leq n$, (\ref{paltanea}) is equivalent to 
$Y_{k:n}\leq_{\rm hr} X_{k:n}$, and also equivalent to $Y_{k:n}\leq_{\rm disp} X_{k:n}$.
\end{theorem}

The statement concerning $\leq_{\rm hr}$ in Theorem~\ref{thm1} confirms a conjecture of P\v{a}lt\v{a}nea (2008).  Theorem~\ref{thm1} is established by considering the star order between $X_{k:n}$ and $Y_{k:n}$.  For random variables $X$ and $Y$ supported on $(0, \infty)$ with distribution functions $F$ and $G$ respectively, we say $X$ is smaller than $Y$ in the star order, denoted by $X\leq_* Y$ (or $F\leq_* G$), if $G^{-1}F(x)/x$ is increasing in $x>0$, where $G^{-1}$ denotes the right continuous inverse of $G$.  Equivalently, $X\leq_* Y$ if and only if, for each $c>0$, $F(cx)$ crosses $G(x)$ at most once, and from below, as $x$ increases on $(0, \infty)$.  For definitions and properties of various stochastic orders including $\leq_*$, see Shaked and Shanthikumar (2007). 

Kochar and Xu (2009) showed $Y_{n:n}\leq_* X_{n:n}$, a key result that allowed Xu and Balakrishnan (2012) to obtain stochastic comparison results for the ranges of heterogeneous exponential samples; see Genest et al. (2009) for related work.  Our Theorem~\ref{thm2} confirms the conjecture of Xu and Balakrishnan (2012) (see also Balakrishnan and Zhao (2013), Open Problem 5).  A special case of our Theorem~\ref{thm2} was obtained by Kochar and Xu (2011) for multiple-outlier models. 

\begin{theorem}
\label{thm2}
For each $1\leq k\leq n$, we have $Y_{k:n}\leq_* X_{k:n}$.
\end{theorem}
The star ordering is a strong variability ordering.  It implies the Lorenz ordering $\leq_{\rm L}$, which in turn implies the ordering of the coefficients of variation.  Our Theorem~\ref{thm2} is a strengthening of the main result of Da et al. (2014) who showed $Y_{k:n}\leq_{\rm L} X_{k:n}$. 

Theorems~\ref{thm1} and \ref{thm2} can be extended from order statistics to general spacings.  

\begin{corollary}
\label{coro1}
For $1\leq m< k \leq n$ we have $Y_{k:n} - Y_{m:n} \leq_* X_{k:n} - X_{m:n}$.
\end{corollary}

\begin{proposition}
\label{prop1}
For $1\leq m < k\leq n$ we have $Y_{k:n} - Y_{m:n} \leq_{\rm order} X_{k:n} - X_{m:n}$ if and only if 
\begin{equation}
\label{char}
{n-m\choose k-m} \gamma^{k-m} \geq \sum_{\mathbf{r}} s_{k-m}^{[\mathbf{r}]}(\lambda) \prod_{j=1}^m \frac{ \lambda_{r_j}}{\Lambda -\sum_{i=1}^{j-1} \lambda_{r_i}} 
\end{equation}
where $\mathbf{r} = (r_1, \ldots, r_m)$, the outer sum is over all permutations of $\{1, \ldots, n\}$ using $m$ at a time, $\lambda = (\lambda_1, \ldots, \lambda_n)$, $\Lambda=\sum_{i=1}^n \lambda_i$, $s_{k-m}^{[\mathbf{r}]}(\lambda) = s_{k-m}(\lambda\backslash\{\lambda_{r_1}, \ldots,\lambda_{r_m}\})$, 
and $\leq_{\rm order}$ is any of $\leq_{\rm st}, \leq_{\rm hr}$ or $\leq_{\rm disp}$. 
\end{proposition} 

In the special case of $m=1$ and $k=n$, which corresponds to comparing the sample ranges, the condition (\ref{char}) reduces to 
$$\gamma \geq \left(\frac{\prod_{i=1}^n \lambda_i}{\Lambda / n} \right)^{1/(n-1)},$$
and we recover Theorem~4.1 of Xu and Balakrishnan (2012).  In the case of ordinary spacings, that is, $k=m+1$, Proposition~\ref{prop1} can also be derived using the log-concavity arguments of Yu (2009).  As noted by P\v{a}lt\v{a}nea (2011), however, it is difficult to implement such an argument in general. 

In Section~2 we prove Theorem~\ref{thm2}, from which Theorem~\ref{thm1}, Corollary~\ref{coro1} and Proposition~\ref{prop1} are deduced.  We refer to P\v{a}lt\v{a}nea (2008) for numerical illustrations and Da et al. (2014) for potential applications of these results.  It would be interesting to see if results similar to Theorem~\ref{thm1} and Proposition~\ref{prop1} can be obtained for the likelihood ratio order, for which Theorem~\ref{thm2}, based on the star order, is not helpful. 

\section{Derivation of main results}
We need the following closure property of the star order with respect to mixtures. 
\begin{lemma}[Lemma~3.1 of Xu and Balakrishnan (2012)]
\label{lem1}
Let $F$ be a distribution function with density $f$ supported on $(0, \infty)$ such that $f(e^x)$ is log concave in $x\in\mathbb{R}$.  Let $G_1,\ldots, G_n$ be distribution functions with densities $g_1,\ldots, g_n$ supported on $(0,\infty)$ such that $F\leq_* G_i$ for each $i=1,\ldots, n$.  Then $F\leq_* \sum_{i=1}^n p_i G_i$ for $p_i>0$ such that $\sum_{i=1}^n p_i =1$. 
\end{lemma}

We also need the following representation of order statistics from heterogeneous exponential samples. 
\begin{lemma}[P\v{a}lt\v{a}nea (2011), Theorem~4.1]
\label{lem2}
Let $S=\{X_1, \ldots, X_n\}$ be a set of $n>1$ independent exponential random variables with rates $\lambda_1,\ldots, \lambda_n$ respectively.  Denote $S^{[i]}=S \backslash \{X_i\}$ and let $X_{j:(n-1)}^{[i]}$ be the $j$th order statistic from $S^{[i]}$ with distribution function $F_{j:(n-1)}^{[i]}$ for $j\leq n-1$.  Then for $k\leq n-1$ the $(k+1)$th order statistic $X_{(k+1):n}$ from $S$ is the sum of two independent random variables: $X_{1:n}$, which has an exponential distribution with rate $\Lambda = \sum_{j=1}^n \lambda_j$, and a mixture of order statistics with distribution function $\Lambda^{-1} \sum_{i=1}^n \lambda_i F_{k:(n-1)}^{[i]}$. 
\end{lemma}

Finally, we need some stochastic comparison results concerning convolutions of gamma variables.  For independent random variables $Z_1,\ldots, Z_n\sim {\rm gamma}(\alpha)$ let $F_\theta$ denote the distribution function of the weighted sum $\sum_{i=1}^n \theta_i Z_i$ where $\theta=(\theta_1,\ldots, \theta_n)$ is a vector of positive weights.  We use $\prec_w$ to denote weak sub-majorization (Marshall et al. (2009)).  The following Lemma is a consequence of Theorem~1 of Yu (2011). 

\begin{lemma}
\label{lem.yu}
For $\alpha >0$, if $\log \eta \prec_w \log \theta$ then $F_\eta\leq_{\rm st} F_\theta$. 
\end{lemma}

Lemma~\ref{lem3} asserts a unique crossing between the distribution functions of two weighted sums of iid gamma variables when the weights form a special configuration.  Lemma~\ref{lem3} is an important step in Yu's (2016) investigation of the unique crossing conjecture of Diaconis and Perlman (1990); it is also a key tool in deriving our main results (we only need the $\alpha=1$ case). 

\begin{lemma}[Theorem~1 of Yu (2016)]
\label{lem3}
Suppose $\alpha \geq 1$.  Suppose $0<\theta_1\leq \cdots \leq \theta_n$ and $\eta_1\leq \cdots \leq \eta_n$ and (a) there exists $2\leq k\leq n$ such that $\theta_i<\eta_i$ for $i< k$ and $\theta_i > \eta_i$ for $i\geq k$;
(b) $\prod_{i=1}^n \eta_i > \prod_{i=1}^n \theta_i$.  Then there exists $x_0\in (0, \infty)$ such that $F_\eta(x) < F_\theta(x)$ for $x\in (0, x_0)$ and the inequality is reversed for $x> x_0$.
\end{lemma}

\begin{proof}[Proof of Theorem~\ref{thm2}]
Let us use induction.  The $k=1$ case is trivial.  Suppose the claim holds for a certain $k\geq 1$ and all $n\geq k$.  We shall prove that it holds for $k+1$ and all $n\geq k+1$.  Assume $\lambda_i,\, i=1,\ldots,n,$ are not all equal, and let $F_{k:n}^{(\tau)}$ denote the distribution function of the $k$th order statistic from a sample of $n$ iid exponential variables with rate $\tau$.  In the notation of Lemma~\ref{lem2}, the induction hypothesis yields $F_{k:(n-1)}^{(\tau)} \leq_* F_{k:(n-1)}^{[i]}$ for $\tau>0,\ n\geq k+1,$ and $i=1,\ldots, n$.  As noted by Xu and Balakrishnan (2012), $F_{k:(n-1)}^{(\tau)}$ has a density $f$ such that $f(e^x)$ is log-concave.  By Lemma~\ref{lem1}, we have 
\begin{equation}
\label{star}
F_{k:(n-1)}^{(\tau)} \leq_* \sum_{i=1}^n \frac{\lambda_i}{\Lambda} F_{k:(n-1)}^{[i]}.
\end{equation}
As noted by Da et al. (2014), when $x\downarrow 0$ we have
$$F_{k:(n-1)}^{[i]}(x) = s_k^{[i]}(\mathbf{\lambda}) x^k + o(x^k),\quad F_{k:(n-1)}^{(\tau)} = {n-1 \choose k} \tau^k x^k + o(x^k),$$
where $\mathbf{\lambda} = (\lambda_1, \ldots, \lambda_n)$ and $s_k^{[i]}(\mathbf{\lambda}) = s_k(\mathbf{\lambda}\backslash \{\lambda_i\})$.  Thus, if ${n-1 \choose k} \tau^k \geq \sum_{i=1}^n \lambda_i s_k^{[i]}(\mathbf{\lambda})/\Lambda $, or equivalently $\tau\geq \tau^*$ with $\tau^* = (n s_{k+1}(\mathbf{\lambda})/\Lambda /{n\choose k+1})^{1/k}$, 
then 
\begin{equation}
\label{st}
F_{k:(n-1)}^{(\tau)} \leq_{\rm st} \sum_{i=1}^n \frac{\lambda_i}{\Lambda} F_{k:(n-1)}^{[i]}.
\end{equation}
Analogous considerations as $x\to\infty$ (P\v{a}lt\v{a}nea, 2011) reveal that if $\tau \leq \tau_*\equiv \sum_{i=1}^{n-k} \lambda_{(i)}/(n-k)$, where $\lambda_{(1)}, \ldots, \lambda_{(n)}$ are $\lambda$ arranged in increasing order, then (\ref{st}) holds with the direction of $\leq_{\rm st}$ reversed.  For $\tau\in (\tau_*, \tau^*)$, star ordering implies that there exists $x_0>0$ such that 
\begin{equation}
\label{one.cross}
F_{k:(n-1)}^{(\tau)}(x)\leq \sum_{i=1}^n \frac{\lambda_i}{\Lambda} F_{k:(n-1)}^{[i]}(x),\quad x\in (0, x_0),
\end{equation}
and the inequality is reversed if $x\in (x_0, \infty)$.  One can show strict inequality, that is, the crossing point is unique.  Indeed, if there exists another $\tilde{x}\in (0, x_0)$, for example, such that equality holds in (\ref{one.cross}) at $x=\tilde{x}$, then a slight increase in $\tau$ would produce at least two crossings, near $x_0$ and $\tilde{x}$, respectively (equality cannot hold for all $x\in [\tilde{x}, x_0]$ unless the two distributions are identical, because these distribution functions can be written as linear combinations of exponential functions and are therefore analytic).  

In view of Lemma~\ref{lem2}, we can convolve both sides of (\ref{star}) with an exponential with rate $\Lambda$ and obtain that ${\rm expo}(\Lambda) * F_{k:(n-1)}^{(\tau)}$ crosses $F_{(k+1):n}$ (the distribution function of $X_{(k+1):n}$) exactly once, and from below, for $\tau\in (\tau_*, \tau^*)$.  That there is at most one crossing follows from variation diminishing properties of TP2 kernels (Karlin 1968).  Upon close inspection there is exactly one crossing at a unique point.  In particular, because $\Lambda > (n-k) \tau> (n-k) \tau_*$, convolving with ${\rm expo}(\Lambda)$ cannot reverse the sign of $F_{k:(n-1)}^{(\tau)}(x) - \Lambda^{-1} \sum_{i=1}^n \lambda_i F_{k:(n-1)}^{[i]}(x)$ as $x\to\infty$.  We then recognize the crossing point, denoted by $x_*$, as a decreasing, continuous function of $\tau$, because $F_{k:(n-1)}^{(\tau)}$ stochastically decreases in $\tau$.  Furthermore, the above analysis at the end points $\tau_*$ and $\tau^*$ shows that $x_*(\tau)\uparrow \infty$ as $\tau\downarrow \tau_*$ and $x_*(\tau)\downarrow 0$ as $\tau\uparrow \tau^*$.  

The distribution of $Y_{(k+1):n}$ is the convolution 
$$F_{(k+1):n}^{(\gamma)}= {\rm expo}(n\gamma) * {\rm expo}((n-1)\gamma) *\cdots * {\rm expo}((n-k)\gamma) .$$ 
Suppose $F_{(k+1): n}^{(\gamma)}$ crosses $F_{(k+1):n}$ at some $x^*>0$.  Then we can choose $\tau\in (\tau_*, \tau^*)$ such that $x_*(\tau) = x^*$, that is, $ {\rm expo}(\Lambda) * F_{k:(n-1)}^{(\tau)}$ crosses $F_{(k+1):n}$ at exactly $x^*$, from below.  If $\gamma \geq \Lambda/n$ then Maclaurin's inequality yields $\gamma\geq \tau^*> \tau$, which implies that 
\begin{equation}
\label{no.cross}
F_{(k+1):n}^{(\gamma)} \leq_{\rm st} {\rm expo}(\Lambda) * F_{k:(n-1)}^{(\tau)},
\end{equation}
contradicting the existence of $x^*$.  If $\gamma\leq \tau$ then the inequality (\ref{no.cross}) is reversed, and there is again no crossing.  Thus we must have $\gamma\in (\tau, \Lambda/n)$.  If $\prod_{i=0}^k ((n-i)\gamma) \geq \Lambda \prod_{i=1}^k ((n-i) \tau)$ then one can show (the $\log$ applies element-wise)
$$-\log(n\gamma, (n-1)\gamma,\ldots, (n-k)\gamma)\prec_w -\log(\Lambda, (n-1)\tau, \ldots, (n-k)\tau),
$$
which again implies (\ref{no.cross}) by Lemma~\ref{lem.yu}.  For the remaining case, $\prod_{i=0}^k ((n-i)\gamma) < \Lambda \prod_{i=1}^k ((n-i) \tau)$, the conditions of Lemma~\ref{lem3} are satisfied.  It follows that $F_{(k+1):n}^{(\gamma)}$ crosses ${\rm expo}(\Lambda) * F_{k:(n-1)}^{(\tau)}$ exactly once, from below, at the same point $x^*$.  We deduce that $F_{(k+1):n}^{(\gamma)}$ crosses $F_{(k+1):n}$ exactly once, from below, at $x^*$.  Since $\gamma>0$ is arbitrary, we have $Y_{(k+1):n}\leq_* X_{(k+1):n}$. 
\end{proof}
\begin{proof}[Proof of Theorem~\ref{thm1}]
Theorem~\ref{thm1} can be deduced from Theorem~\ref{thm2} as discussed by Xu and Balakrishnan (2012).  Specifically, given $Y_{k:n}\leq_* X_{k:n}$, we have the equivalence 
$$Y_{k:n}\leq_{\rm st} X_{k:n}\Longleftrightarrow Y_{k:n}\leq_{\rm disp} X_{k:n}.$$
Since $Y_{k:n}$ has increasing failure rate, dispersive ordering implies hazard rate ordering. 
\end{proof}
\begin{proof}[Proof of Corollary~\ref{coro1}]
Extending Lemma~\ref{lem2} we can write the distribution function of $X_{k:n}-X_{m:n}$ as 
\begin{equation}
\label{spacing}
\sum_{\mathbf{r}} \prod_{j=1}^m \frac{\lambda_{r_j}}{\Lambda -\sum_{i=1}^{j-1} \lambda_{r_i}} F_{(k-m):(n-m)}^{[\mathbf{r}]},
\end{equation}
where the notation is the same as in the statement of Proposition~\ref{prop1}, and $F_{(k-m):(n-m)}^{[\mathbf{r}]}$ denotes the distribution function of the $(k-m)$th order statistics of $\{X_1, \ldots, X_n\}\backslash\{X_{r_1}, \ldots, X_{r_m}\}$.  The distribution of $Y_{k:n} - Y_{m:n}$ is simply $F_{(k-m):(n-m)}^{(\gamma)}$.  The claim therefore follows from Theorem~\ref{thm2} and Lemma~\ref{lem1}. 
\end{proof}

\begin{proof}[Proof of Proposition~\ref{prop1}]
Given the star ordering, we can establish the characterization for $\leq_{\rm st}$ by examining the distribution function (\ref{spacing}) near $x=0$ (we did this for $m=1$ in the proof of Theorem~\ref{thm2}).  Characterizations for $\leq_{\rm hr}$ and $\leq_{\rm disp}$ follow as in the proof of Theorem~\ref{thm1}.
\end{proof}

\end{document}